\documentclass{amsart}

\usepackage{amssymb,amscd,amsmath,amsthm}

\theoremstyle{plain}
\newtheorem{thm}{Theorem}[section]
\newtheorem{theorem}[thm]{Theorem}

\theoremstyle{definition}
\newtheorem{defn}[thm]{Definition}
\newtheorem{conj}[thm]{Conjecture}

\theoremstyle{remark}
\newtheorem{rem}{Remark}

\newtheorem*{remark*}{Remark}

\newcommand{\cF}{{\mathcal{F}}}

\newcommand{\cL}{{\mathcal{L}}}

\newcommand{\cS}{{\mathcal{S}}}

        \newcommand{\field}[1]{{\mathbb{#1}}}
        
        \newcommand{\ZZ}{\field{Z}}
        
        \newcommand{\RR}{\field{R}}
        
\newcommand{\TT}{\field{T}}

\newcommand{\Tr}{\operatorname{Tr}}
\newcommand{\tr}{\operatorname{tr}}

\begin{document}

\title[Adiabatic limits and noncommutative Weyl formula]{Adiabatic limits and noncommutative Weyl formula}
\author[Y. A. Kordyukov]{Yuri A. Kordyukov}
\address{Institute of Mathematics\\
         Russian Academy of Sciences\\
         112~Chernyshevsky str.\\ 450008 Ufa\\ Russia} \email{yurikor@matem.anrb.ru}

\thanks{Supported by the Russian Foundation of Basic Research
(grant no. 09-01-00389)}

\begin{abstract}
We discuss asymptotic behavior of the eigenvalue distribution of the
differential form Laplacian on a Riemannian foliated manifold when
the metric on the ambient manifold is blown up in directions normal
to the leaves (in the adiabatic limit). Motivated by analogies with
semiclassical spectral asymptotics, we use ideas and notions of
noncommutative geometry to suggest a conjectural formula for the
eigenvalue distribution in the adiabatic limit, which we call
noncommutative Weyl formula. We review known results and discuss the
correctness of the noncommutative Weyl formula in each case.
\end{abstract}
\maketitle

\section*{Introduction}
In this paper we discuss a particular asymptotic spectral problem
for the Laplace operator on a closed Riemannian foliated manifold,
namely, the asymptotic behavior of its eigenvalue distribution in
the adiabatic limit.

Let $(M,{\mathcal F})$ be a closed foliated manifold, $\dim M = n$,
$\dim {\mathcal F} = p$, $p+q=n$, endowed with a Riemannian metric
$g$. Then we have a decomposition of the tangent bundle to $M$ into
a direct sum $TM=F\oplus H$, where $F=T{\mathcal F}$ is the tangent
bundle to $\cF$ and $H=F^{\bot}$ is the orthogonal complement of
$F$, and the corresponding decomposition $g=g_{F}+g_{H}$ of the
metric into the sum of tangential and transverse components. Define
a one-parameter family $g_{\varepsilon}$ of Riemannian metrics on
$M$ by
\begin{equation}\label{e:gh}
g_{\varepsilon}=g_{F} + {\varepsilon}^{-2}g_{H}, \quad \varepsilon >
0.
\end{equation}
By the adiabatic limit, we mean the asymptotic behavior of
Riemannian manifolds $(M, g_\varepsilon)$ as $\varepsilon\to 0$. In
this form, the notion of the adiabatic limit was introduced by
Witten \cite{Witten85} in the study of the global anomaly. We refer
the reader to a survey paper \cite{bedlewo2-andrey} for some history
and references.

For any $\varepsilon>0$, consider the Laplace operator
$\Delta_{\varepsilon}$ on differential forms defined by the metric
$g_\varepsilon$. It is a second order elliptic differential
operator, which is self-adjoint in the Hilbert space $L^2(M,\Lambda
T^{*}M,g_\varepsilon)$ of square integrable differential forms on
$M$ endowed with the inner product induced by $g_\varepsilon$ and
has a complete orthonormal system of eigenforms. Denote by
$0\leq\lambda_0(\varepsilon)\leq\lambda_1(\varepsilon)\leq\lambda_2(\varepsilon)
\leq\cdots $ the eigenvalues of $\Delta_{\varepsilon}$, taking
multiplicities into account. We will discuss the asymptotic behavior
as $\varepsilon\to 0$ of the trace of $f(\Delta_\varepsilon)$:
\[
\tr
f(\Delta_\varepsilon)=\sum_{i=0}^{+\infty}f(\lambda_i(\varepsilon)),
\]
for any sufficiently nice function $f$, for instance, for $f\in
S(\RR)$.

In order to explain what kind of spectral problems we have, we first
transfer the operators $\Delta_{\varepsilon}$ to the fixed Hilbert
space
\[
L^2\Omega = L^{2}(M, \Lambda T^{*}M, g),
\]
using an isomorphism $\Theta_\varepsilon$ from $L^{2}(M,\Lambda
T^{*}M , g_{\varepsilon})$ to $L^2\Omega$ defined as follows. With
respect to a bigrading on $\Lambda T^{*}M$ given by
\[
\Lambda^{i,j}T^{*}M=\Lambda^{i}H^{*}\otimes \Lambda^{j}F^{*}, \quad
0 \leq i \leq q, \quad 0\leq j \leq p,
\]
we have, for $u \in L^{2}(M,\Lambda^{i,j}T^{*}M , g_{\varepsilon})$,
\begin{equation*}
\Theta_{\varepsilon}u = \varepsilon^{i}u.
\end{equation*}
The operator $\Delta_\varepsilon$ in $L^{2}(M,\Lambda T^{*}M ,
g_{\varepsilon})$ corresponds under the isometry
$\Theta_{\varepsilon}$ to the operator $L_{\varepsilon}=
\Theta_{\varepsilon}\Delta_\varepsilon\Theta_{\varepsilon}^{-1}$ in
$L^2\Omega$.

With respect to the above bigrading of $\Lambda T^*M$, the de Rham
differential $d$ can be written as
\[
d=d_F+d_H+\theta,
\]
where
\begin{enumerate}
\item  $ d_F=d_{0,1}: C^{\infty}(M,\Lambda^{i,j}T^*M)\to
C^{\infty}(M,\Lambda^{i,j+1}T^*M)$ is the tangential de Rham
differential, which is a first order tangentially elliptic
operator, independent of the choice of $g$;
\item $d_H=d_{1,0}: C^{\infty}(M,\Lambda^{i,j}T^*M)\to
C^{\infty}(M,\Lambda^{i+1,j}T^*M)$ is the transversal de Rham
differential, which is a first order transversally elliptic
operator;
\item $\theta=d_{2,-1}: C^{\infty}(M,\Lambda^{i,j}T^*M)\to
C^{\infty}(M,\Lambda^{i+2,j-1}T^*M)$ is a zeroth order
differential operator.
\end{enumerate}
One can show that
\[
d_\varepsilon = \Theta_{\varepsilon}d\Theta_{\varepsilon}^{-1} = d_F
+ \varepsilon d_H + \varepsilon^{2}\theta,
\]
and the adjoint of $d_\varepsilon$ in  $L^2\Omega$ is
\[
\delta_\varepsilon=\Theta_{\varepsilon}\delta
\Theta_{\varepsilon}^{-1} = \delta_F + \varepsilon \delta_H +
\varepsilon^{2}\theta^{*}.
\]
Therefore, for the operator
\begin{equation*}
L_\varepsilon = \Theta_{\varepsilon}\Delta_\varepsilon
\Theta_{\varepsilon}^{-1} = d_\varepsilon\delta_\varepsilon +
\delta_\varepsilon d_\varepsilon,
\end{equation*}
one has
\begin{equation*}
L_\varepsilon =\Delta_F + \varepsilon^2\Delta_H +
\varepsilon^4\Delta_{\theta}+ \varepsilon K_1+\varepsilon^2K_2
+\varepsilon^3K_3,
\end{equation*}
where
\begin{itemize}
  \item $\Delta_F=d_Fd^*_F+d^*_Fd_F$ is the tangential Laplacian;
  \item $\Delta_H=d_Hd^*_H+d^*_Hd_H$ is the transverse Laplacian;
  \item $\Delta_{\theta}=\theta\theta^{*}+ \theta^{*}\theta$ and $K_2 = d_F
\theta^{*}+ \theta^{*} d_F + \delta_F\theta+\theta \delta_F$ are of
zero order;
  \item $K_1 =  d_F \delta_H+ \delta_H d_F +
\delta_Fd_H+d_H \delta_F$ and $K_3 = d_H \theta^{*}+ \theta^{*} d_H
+ \delta_H\theta+\theta \delta_H$ are first order differential
operators.
\end{itemize}

Differential (or pseudodifferential) operators with a small
parameter are usually set up on $\RR^n$, within the Weyl calculus,
under the name of $h$-admissible or semiclassical operators. Recall
that a semiclassical differential operator on $\RR^n$ is the
differential operator $B_h$, depending on a parameter $h>0$, of the
form
\[
B_h=\sum_{j=0}^kh^jB_j(x,hD_x),\quad x\in \RR^n,
\]
where each $B_j(x,\xi)$ is a polynomial in $\xi$:
\[
B_j(x,\xi)=\sum_{|\alpha|\leq k}b_\alpha^{(j)}\xi^\alpha,
\]
and
\[
B_j(x,hD_x)=\sum_{|\alpha|\leq k}b_\alpha^{(j)}D_x^\alpha,
\]
where $\alpha=(\alpha_1,\ldots,\alpha_n)\in \ZZ^n_+$ is a
multi-index, and we use standard notation: $|\alpha|=\alpha_1 +
\ldots + \alpha_n$,
$\xi^\alpha=\xi_1^{\alpha_1}\ldots\xi_n^{\alpha_n}$,
$D_x^\alpha=D_{x_1}^{\alpha_1}\ldots D_{x_n}^{\alpha_n}$,
$D_{x_j}=\frac{1}{i}\frac{\partial}{\partial x_j}$.

The leading term $B_0(x,\xi)$ is called the principal symbol of
$B_h$.

On an arbitrary smooth manifold, the definitions of semiclassical
differential operators and their principal symbols are much more
delicate \cite{Paul-Uribe95}. Nevertheless, there is one particular
important case, when this can be done rather easily, the case of
Schr\"odinger operator.

Let $(X,g)$ be a compact Riemannian manifold, $\dim X=n$. The
Schr\"odinger operator on $X$ is the self-adjoint second order
differential operator $H_h$ acting in $C^\infty(X)$ by the formula
\[
H_h=-h^2\Delta +V,
\]
where $h>0$ is the semiclassical parameter, $\Delta$ is the
Laplace-Beltrami operator associated with the Riemannian metric $g$,
and $V\in C^\infty(X)$ is a real-valued smooth function on $X$ (a
potential), which is identified with the corresponding
multiplication operator. The principal symbol $\sigma(H_h)\in
C^\infty(T^*X)$ of the semiclassical operator $H_h$ is given by
\begin{equation*}
\sigma(H_h)(x,\xi)=g^{T^*X}_x(\xi,\xi)+V(x),\quad (x,\xi)\in T^{*}X,
\end{equation*}
where $g^{T^*X}_x$ is the induced metric in $T^*_xX$.

For any $f\in \cS(\RR)$, the operator $f(H_h)$ is of trace class,
and its asymptotic behavior as $h\to 0$ is described by the
semiclassical Weyl formula:
\begin{equation}\label{e:Weyl}
\operatorname{tr} f(H_h)=\frac{1}{(2\pi h)^{n}}\int_{T^*X}
f(\sigma(H_h)(x,\xi))\,dx\,d\xi+o(h^{-n}),\quad h\rightarrow 0+,
\end{equation}
where $dx\,d\xi$ is the canonical Liouville measure on $T^*X$.

In our case, in a local foliated chart with coordinates $(x,y)\in
\RR^p\times\RR^q$, where the leaves of the foliation are given by
the level sets $y={\rm const}$, the operator $L_\varepsilon$ has the
form
\begin{equation}\label{e:Le}
L_\varepsilon=\sum_{j=0}^k\varepsilon^jA_j(x, y, D_x, \varepsilon
D_y),\quad x\in \RR^p,\quad y\in \RR^q.
\end{equation}
In other words, the parameter $\varepsilon$ enters the coefficients
of derivatives with respect to $y$, and there is no $\varepsilon$ in
the coefficients of the derivatives with respect to $y$.

Asymptotic spectral problems of this kind appeared for the first
time in molecular quantum mechanics in a paper by Born and
Oppenheimer in 1927. Born and Oppenheimer studied molecular bound
states, that is, roughly speaking, eigenfunctions of the
Schr\"odinger operator
\[
H_h=-\frac{h^2}{2m}\Delta_x-\frac{h^2}{2M}\Delta_y+V(x,y)
\]
in $\RR^n=\RR^p_x\times \RR^q_y$, where the $x$ variables describe
the electron motion, the $y$ variables describe the nuclear motion,
$m$ is the electron mass, $M$ is the nuclei mass and $V$ is the
interaction potential. They suggested an approximation for molecular
bound states based on disparity between the electron mass $m$ and
the nuclei mass $M$. This approximation makes use of asymptotic
expansions in the small parameter $\varepsilon$, where
$\varepsilon^4=m/M$.

The first rigorous mathematical paper on Born-Oppenheimer
approximation was published by Seiler in 1973. Mathematical
investigations of Born-Oppenheimer approximation continue very
actively to present day (see, for instance, survey papers
\cite{Belov-Dobr-Tud,HagedornJoye} and references there). Later,
similar problems were also studied in many areas of mathematical and
theoretical physics (see, for instance, \cite{Belov-Dobr-Tud}).

Observe that an operator of the form \eqref{e:Le} can be considered
as an semiclassical differential operator in $\RR^q$ with an
operator-valued symbol. In particular, the asymptotic formula for
the trace of $f(L_\varepsilon)$ for such an operator $L_\varepsilon$
in $\RR^n$ can be written in a form, similar to the semiclassical
Weyl formula \eqref{e:Weyl}. For instance, let us consider the
Schr\"odinger operator
\[
H_\varepsilon=\Delta_X+\varepsilon^2 \Delta_Y+V
\]
in $M=X\times Y$, where $X$ and $Y$ are closed Riemannian manifolds,
$\Delta_X$ and $\Delta_Y$ are the Laplace-Beltrami operators on $X$
and $Y$ respectively, $V\in C^\infty(M)$. The principal symbol of
$H_\varepsilon$ is a smooth function $\sigma(H_\varepsilon)$ on
$T^*Y$, whose values are differential operators on $X$:
\begin{equation}\label{e:principal-op}
\sigma(H_\varepsilon)(y,\eta)=g^{T^*Y}_y(\eta,\eta)+\Delta_X+V_y,\quad
(y,\eta)\in T^{*}Y,
\end{equation}
where $V_y\in C^\infty(X)$ is the restriction of $V$ to $X\times
\{y\}$.

For any $f\in \cS(\RR)$ and $(y,\eta)\in T^{*}Y$, the operator
$f(\sigma(H_\varepsilon)(y,\eta))$ is a smoothing operator on $X$,
so it is a trace class operator. Moreover, one can show that
\[
\int_{T^*Y} \operatorname{tr}
f(\sigma(H_\varepsilon)(y,\eta))\,dy\,d\eta<\infty.
\]
Then we have
\begin{equation}\label{e:Weyl-op}
\operatorname{tr} f(H_\varepsilon)
=\frac{1}{(2\pi\varepsilon)^{q}}\int_{T^*Y} \operatorname{tr}
f(\sigma(H_\varepsilon)(y,\eta))\,dy\,d\eta+o(\varepsilon^{-q}),\quad
\varepsilon \rightarrow 0+.
\end{equation}
This formula was proved, for instance, in \cite{Bal} for the case of
$\RR^n$, using machinery of pseudodifferential operators with
operator-valued symbols.

In the foliation case, such a nice picture holds only locally. The
global structure of a foliation may be very complicated. There may
be no base manifold, leaves may be non-compact, and elliptic
operators along leaves have, in general, continuous spectrum.
Therefore, it is even unclear how a general asymptotic formula for
$\operatorname{tr} f(\Delta_\varepsilon)$ could look like.
Nevertheless, one can write such a conjectural formula, which we
call noncommutative Weyl formula, using ideas and notions of
noncommutative geometry developed by A. Connes. (For the basic
information on noncommutative geometry of foliations, we refer the
reader to \cite{survey,umn-survey} and references therein.)

First, observe that, using the results of \cite{asymp}, one can show
that, for any $f\in \cS(\RR)$, we have the estimate
\[
|\tr f(\Delta_\varepsilon)|=|\tr f(L_\varepsilon)|<
C_1\varepsilon^{-q}, \quad 0<\varepsilon\leq 1,
\]
with some $C_1>0$. Moreover, it is easy to show that the terms
$\varepsilon^2K_2$, $\varepsilon^3K_3$ and
$\varepsilon^4\Delta_{\theta}$ are of lower order in some sense,
and, therefore, they don't contribute to the leading term of the
formula. More precisely, if we denote
\[
\bar{L}_\varepsilon =\Delta_F + \varepsilon K_1 +
\varepsilon^2\Delta_H,
\]
then, for any $f\in \cS(\RR)$, we have (with some $C_2>0$):
\[
|\tr f(L_\varepsilon)-\tr f(\bar{L}_\varepsilon)|
<C_2\varepsilon^{1-q}, \quad 0<\varepsilon\leq 1,
\]
Therefore, we can restrict our considerations by the operator
$\bar{L}_\varepsilon$.

Now we have certain difficulties mentioned above with an appropriate
definition of the principal symbol of $\bar{L}_\varepsilon$, first
of all, with a definition of the principal symbol of the term
$\varepsilon K_1$. We were able find a solution in the case when
$\varepsilon K_1$ is a lower order operator in some sense (this
happens for Riemannian foliations) or when $K_1=0$. In this case,
our considerations are reduced to the operator
\[
\bar{\bar{L}}_\varepsilon =\Delta_F + \varepsilon^2\Delta_H,
\]
Here we observe that the operator $\bar{\bar{L}}_\varepsilon$ has
the form of a Schr\"odinger operator on the leaf space $M/\cF$,
where $\Delta_H$ plays the role of the Laplace operator, and
$\Delta_F$ the role of the operator-valued potential on $M/\cF$.
Using this analogy, we define the principal symbol of this operator
as a second order differential operator $\sigma(\Delta_\varepsilon)$
on the conormal bundle $N^*\cF$ of $\cF$, which is tangentially
elliptic with respect to a natural foliation $\cF_N$ on $N^*\cF$. We
will consider the operator $\sigma(\Delta_\varepsilon)$ as a family
of self-adjoint elliptic operators along the leaves of ${\mathcal
F}_N$. Then, for any $f\in S({\mathbb R})$, one can define
$f(\sigma(\Delta_\varepsilon))$ as an element of the twisted
$C^*$-algebra $C^*(N^*\cF, {\mathcal F}_N,\pi^{*}\Lambda T^{*}M)$
associated with the foliation $(N^*\cF, {\mathcal F}_N)$. The leaf
space $N^*{\cF}/\cF_N$ can be considered as the cotangent bundle to
$M/\cF$, and the algebra $C^*(N^*\cF, {\mathcal F}_N,\pi^{*}\Lambda
T^{*}M)$ can be viewed as the noncommutative analogue of the algebra
of continuous vector-valued differential forms on this singular
space.

The foliation $\cF_N$ has a natural transverse symplectic structure.
The corresponding canonical transverse Liouville measure is holonomy
invariant and, by noncommutative integration theory \cite{Co79},
determines the trace $\operatorname{tr}_{{\mathcal F}_N}$ on the
$C^*$-algebra $C^*(N^*\cF, {\mathcal F}_N,\pi^{*}\Lambda T^{*}M)$.
The trace $\operatorname{tr}_{{\mathcal F}_N}$ is the noncommutative
analogue of the integral over the leaf space $N^*{\cF}/\cF_N$ with
respect to the transverse Liouville measure.

Replacing in the formula (\ref{e:Weyl-op}) the integration over
$T^{*}X$ and the operator trace $\operatorname{Tr}$ by the trace
$\operatorname{tr}_{{\mathcal F}_N}$ and the principal symbol
$\sigma(H_\varepsilon)$ by the principal
$\sigma(\Delta_\varepsilon)$, we suggest the following
noncommutative analogue of this formula.

\begin{conj}[The noncommutative Weyl formula]
For any $f\in S({\mathbb R})$, the following asymptotic formula
holds:
\begin{equation}\label{e:ncWeyl}
\operatorname{tr} f(\Delta_{\varepsilon})
=\frac{1}{(2\pi\varepsilon)^{q}} \operatorname{tr}_{{\mathcal F}_N}
f(\sigma(\Delta_\varepsilon)) +o(\varepsilon^{-q}),\quad
\varepsilon\rightarrow 0.
\end{equation}
\end{conj}

It is a remarkable fact that such a noncommutative Weyl formula was
rigorously proved for Riemannian foliations in \cite{adiab} (see
also \cite{asymp}). We will explain the details and give two
concrete examples in Section~\ref{s;Riem}.

It is very little known about the problem in question for
non-Riemannian foliations. There are only two computations in some
particular cases \cite{matsb,MZ,SMZ}. In all these cases, we have
$K_1=0$, so it is natural to ask whether the noncommutative Weyl
formula \eqref{e:ncWeyl} holds. We show that the formula obtained in
\cite{matsb} can be written in the form \eqref{e:ncWeyl}, while the
formula obtained in \cite{MZ,SMZ} seems to be not compatible with
the formula \eqref{e:ncWeyl}. These are main new results of the
paper. They will be discussed in Section~\ref{s:nonRiem}.

\section{Riemannian foliations}\label{s;Riem} \subsection{General results}
Recall (see, for instance, \cite{Molino}) that a foliation
${\mathcal F}$ is called Riemannian, if there exists a Riemannian
metric $g$ on $M$ such that the induced metric $g_\tau$ on the
normal bundle $\tau=TM/F$ is holonomy invariant, or, equivalently,
in any foliated chart $\phi:U\to I^p\times I^q$ with local
coordinates $(x,y)$, the restriction $g_H$ of $g$ to $H=F^{\bot}$ is
written in the form $$ g_H=\sum_{\alpha,\beta=1}^q
g_{\alpha\beta}(y)\theta^{\alpha} \theta^{\beta}, $$ where
$\theta^{\alpha}\in H^*$ is the 1-form, corresponding to the form
$dy^\alpha$ under the isomorphism $H^*\cong T^*{\RR}^q$, and
$g_{\alpha\beta}(y)$ depend only on the transverse variables $y\in
\RR^q$. Such a Riemannian metric is called bundle-like.

Suppose that ${\mathcal F}$ is a Riemannian foliation and $g$ is a
bundle-like metric. Then the key observation is that, in this case,
the first order differential operator $K_1$ is a {\em leafwise}
differential operator. Using this fact, one can show that we can
reduce our considerations to the operator
\[
\bar{\bar{L}}_\varepsilon=\Delta_F + \varepsilon^2\Delta_H.
\]
More precisely, for any $f\in \cS(\RR)$, we have (with some $C>0$):
\[
|\tr f(L_\varepsilon)-\tr f(\bar{\bar{L}}_\varepsilon)|<
C\varepsilon^{1-q}, \quad 0<\varepsilon\leq 1.
\]

Now we explain how the asymptotic formula for the trace of
$f(\Delta_\varepsilon)$ in the adiabatic limit can be written in the
form \eqref{e:ncWeyl}. We start with a definition of the principal
symbol.

Denote by $\pi :N^*\cF\to M$ the conormal bundle to ${\cF}$. It is a
well-known fact in foliation theory (cf., for instance,
\cite{Molino,survey}) that $N^*\cF$ carries a natural foliation
${\cF}_N$, which we will call the linearized foliation. The leaves
of $\cF_N$ can be described as follows. First of all, observe that,
for any piecewise smooth leafwise path $\gamma : [0,1]\to M$ such
that $\gamma(0)=x$ and $\gamma(1)=y$, sliding along the leaves of
the foliation determines a linear map $dh^*_{\gamma}: N^*_y{\cF}\to
N^*_x{\cF}$, called the linear holonomy map. Then the leaf through
$\nu\in N^*{\cF}$ consists of all $dh_{\gamma}^{*}(\nu)\in N^*{\cF}$
with $\gamma : [0,1]\to M$ such that $\gamma(1)=\pi(\nu)$.

The leaves of ${\cF}_N$ are transverse to the fibers of $\pi$, and
the restriction of $\pi$ to any leaf of ${\cF}_N$ is a covering over
a leaf of $\cF$. Therefore, the tangential Laplacian $\Delta_F$ can
be lifted to a tangentially elliptic (relative to ${\mathcal F}_N$)
operator $\Delta_{{\mathcal F}_N}$ in $C^{\infty}(N^*{\mathcal
F},\pi^{*} \Lambda T^{*}M)$. Let $g^N \in C^\infty(N^*\cF)$ be the
fiberwise Riemannian metric on $N^*{\mathcal F}$ induced by the
metric on $M$. Denote by $g^N$ the multiplication operator in
$C^{\infty}(N^*{\mathcal F},\pi^{*} \Lambda T^{*}M)$ by $g^N$. The
principal symbol of $\Delta_\varepsilon$ is a tangentially elliptic
operator in $C^{\infty}(N^*{\mathcal F},\pi^{*} \Lambda T^{*}M)$
given by (cf. \eqref{e:principal-op})
\begin{equation}\label{e:principal}
\sigma(\Delta_\varepsilon) = g^N(\eta,\eta) + \Delta_{{\mathcal
F}_N}.
\end{equation}

To define the trace of $f(\sigma(\Delta_\varepsilon))$, we will use
noncommutative integration theory developed by A. Connes in
\cite{Co79}.

Let $G$ be the holonomy groupoid of $\cF$. Let us briefly recall its
definition. Denote by $\sim_h$ the equivalence relation on the set
of piecewise smooth leafwise paths $\gamma:[0,1] \rightarrow M$,
setting $\gamma_1\sim_h \gamma_2$ if $\gamma_1$ and $\gamma_2$ have
the same initial and final points and the same holonomy maps. The
holonomy groupoid $G$ is the set of $\sim_h$ equivalence classes of
leafwise paths. The set of units of $G$ is $G^{(0)}=M$. $G$ is
equipped with the source and the range maps $s,r:G\rightarrow M$
defined by $s(\gamma)=\gamma(0)$ and $r(\gamma)=\gamma(1)$. The
multiplication $\gamma_1\gamma_2\in G$ of $\gamma_1\in G$ and
$\gamma_2\in G$ is defined by the concatenation of paths. It is
defined if and only if $s(\gamma_1)=r(\gamma_2)$. Recall also that,
for any $x\in M$, the set $G^x=\{\gamma\in G : r(\gamma)=x\}$ is the
covering of the leaf $L_x$ through the point $x$, associated with
the holonomy group of the leaf. We will identify any $x\in
G^{(0)}=M$ with the element of $G$ given by the constant path
$\gamma(t)=x, t\in [0,1]$.

The holonomy groupoid $G_{{\mathcal F}_N}$ of the linearized
foliation ${\mathcal F}_N$ can be described as the set of all
$(\gamma,\nu)\in G\times {N}^*{\mathcal F}$ such that
$r(\gamma)=\pi(\nu)$. The source map $s_N:G_{{\mathcal
F}_N}\rightarrow {N}^*{\mathcal F}$ and the range map
$r_N:G_{{\mathcal F}_N}\rightarrow {N}^*{\mathcal F}$ are defined as
$s_N(\gamma,\nu)=dh_{\gamma}^{*}(\nu)$ and $r_N(\gamma,\nu)=\nu$.
The product of $(\gamma_1,\nu_1)\in G_{{\mathcal F}_N}$ and
$(\gamma_2,\nu_2)\in G_{{\mathcal F}_N}$ is defined if
$\nu_2=dh_{\gamma_1}^*(\nu_1)$, and, under this condition, it is
given by
$(\gamma_1,\nu_1)(\gamma_2,dh_{\gamma_1}^*(\nu_1))=(\gamma_1\gamma_2,\nu_1)$.

Denote by ${\mathcal L}(\pi^*\Lambda T^{*}M)$ the vector bundle on
$G_{\cF_N}$, whose fiber at a point $(\gamma,\nu)\in G_{\cF_N}$ is
the space of linear maps $ (\pi^*\Lambda T^{*}M)_{s_N(\gamma,\nu)}
\to (\pi^*\Lambda T^{*}M)_{r_N(\gamma,\nu)}$. There is a standard
way (due to Connes \cite{Co79}) to introduce the structure of
involutive algebra on the space $C^{\infty}_{c}(G_{{\mathcal F}_N},
{\mathcal L}(\pi^*\Lambda T^{*}M))$ of smooth, compactly supported
sections of ${\mathcal L}(\pi^*\Lambda T^{*}M)$. For any $\nu \in
N^*\cF$, this algebra has a natural representation $R_\nu$ in the
Hilbert space $L^2(G^\nu_{\cF_N}, s^*_N(\pi^*\Lambda T^{*}M))$ that
determines its embedding to the $C^*$-algebra of all bounded
operators in $L^2(G_{\cF_N}, s^*_N(\pi^*\Lambda T^{*}M))$.

Let $\lambda_L$ denote the Riemannian volume form on a leaf $L$
given by the induced metric, and, for any $x\in M$, let
$\lambda^{x}$ denote the lift of $\lambda_{L_x}$ via the holonomy
covering map $s:G^x\to L_x$. For any $k\in
C^{\infty}_{c}(G_{{\mathcal F}_N}, {\mathcal L}(\pi^*\Lambda
T^{*}M))$, the action of $R_\nu(k)$ on $\zeta\in L^2(G^\nu_{\cF_N},
s^*_N(\pi^*\Lambda T^{*}M))$ is given by
\begin{equation}\label{e:Rnu}
R_\nu(k)\zeta(\gamma,\nu)=\int_{G^x}
k((\gamma,\nu)^{-1}(\gamma_1,\nu))
\zeta(\gamma_1,\nu)d\lambda^x(\gamma_1), \quad r(\gamma)=x.
\end{equation}
Taking the closure of the image of this embedding, we get a
$C^*$-algebra, called the twisted foliation $C^*$-algebra and
denoted by $C^*(N^*\cF, {\mathcal F}_N,\pi^{*}\Lambda T^{*}M)$.

The foliation $\cF_N$ has a natural transverse symplectic
structure, which can be described as follows. Consider a foliated
chart $\varkappa: U\subset M\rightarrow I^p\times I^q$ on $M$ with
coordinates $(x,y)\in I^p\times I^q$ ($I$ is the open interval
$(0,1)$) such that the restriction of $\cF$ to $U$ is given by the
sets $y={\rm const}$. One has the corresponding coordinate chart
in $T^*M$ with coordinates denoted by $(x,y,\xi,\eta)\in I^p\times
I^q\times \RR^p\times \RR^q$. In these coordinates, the
restriction of the conormal bundle $N^*{\mathcal  F}$ to $U$ is
given by the equation $\xi=0$. So we have a coordinate chart
$\varkappa_n : U_1\subset N^*{\mathcal  F} \longrightarrow
I^p\times I^q\times \RR^q$ on $N^*{\mathcal  F}$ with the
coordinates $(x,y,\eta)\in I^p\times I^q\times \RR^q$. Indeed, the
coordinate chart $\varkappa_n$ is a foliated coordinate chart for
${\mathcal  F}_N$, and the restriction of ${\mathcal F}_N$ to
$U_1$ is given by the level sets $y= {\rm const}, \eta={\rm
const}$. The transverse symplectic structure for $\cF_N$ is given
by the transverse two-form $\sum_jdy_j\wedge d\eta_j$.

The corresponding canonical transverse Liouville measure $dy\,d\eta$
is holonomy invariant and, by noncommutative integration theory
\cite{Co79}, defines the trace $\operatorname{tr}_{{\mathcal F}_N}$
on the $C^*$-algebra $C^*(N^*\cF, {\mathcal F}_N,\pi^{*}\Lambda
T^{*}M)$. Combining the Riemannian volume forms along the leaves of
$\cF_N$ and the transverse Liouville measure $dy\,d\eta$, we get a
volume form $d\nu$ on $N^*\cF$. For any $k\in
C^{\infty}_{c}(G_{{\mathcal F}_N}, {\mathcal L}(\pi^*\Lambda
T^{*}M))$, its trace is given by the formula
\begin{equation}\label{e:trfn}
\operatorname{tr}_{{\mathcal F}_N}(k)=\int_{N^*\cF}k(\nu)d\nu.
\end{equation}

Consider $\sigma(\Delta_\varepsilon)$ as a family of elliptic
operators along the leaves of the foliation ${\mathcal F}_N$ and
lift these operators to the holonomy coverings of the leaves. For
any $\nu\in N^*\cF$, we get a formally self-adjoint uniformly
elliptic operator $\sigma(\Delta_\varepsilon)_\nu$ in
$C^\infty(G^\nu_{\cF_N}, s^*_N(\pi^*\Lambda T^{*}M))$, which is
essentially self-adjoint in the Hilbert space $L^2(G^\nu_{\cF_N},
s^*_N(\pi^*\Lambda T^{*}M))$. For any $f\in S({\mathbb R})$, the
family $\{f(\sigma(\Delta_\varepsilon)_\nu), \nu\in N^*\cF\}$
defines an element $f(\sigma(\Delta_\varepsilon))$ of the
$C^*$-algebra $C^*(N^*\cF, {\mathcal F}_N,\pi^{*}\Lambda T^{*}M)$:
\[
f(\sigma(\Delta_\varepsilon)_\nu)=
R_\nu(f(\sigma(\Delta_\varepsilon))), \quad \nu \in N^*\cF.
\]
One can show that
\[
\operatorname{tr}_{{\mathcal F}_N}
f(\sigma(\Delta_\varepsilon))<\infty.
\]

\begin{theorem}[\cite{adiab}]
\label{ad:main} Let $(M,{\mathcal F})$ be a Riemannian foliation
equipped with a bundle-like Riemannian metric $g$. For any $f\in
S({\mathbb R})$, the asymptotic formula holds:
\begin{equation}\label{e:adiab}
\operatorname{tr} f(\Delta_{\varepsilon})
=\frac{1}{(2\pi\varepsilon)^{q}} \operatorname{tr}_{{\mathcal F}_N}
f(\sigma(\Delta_\varepsilon)) +o(\varepsilon^{-q}),\quad
\varepsilon\rightarrow 0.
\end{equation}
\end{theorem}

It should be noted that the formula \eqref{e:adiab} (or, equally,
\eqref{e:ncWeyl}) makes sense for an arbitrary foliation $\cF$, not
necessarily Riemannian. The only difference is that in general the
groupoid $G_{\cF_N}$ does not coincide with the holonomy groupoid of
$\cF_N$, but we can use it to write the formula. A very important
feature of the case of Riemannian foliation is the fact that, in
this case, the terms $g^N(\eta,\eta)$ and $\Delta_{{\mathcal F}_N}$
in~\eqref{e:principal} commute.

The formula (\ref{e:adiab}) can be rewritten in terms of the
spectral data of leafwise Laplace operators. We will formulate the
corresponding result for the spectrum distribution function
\[
N_\varepsilon(\lambda)=\sharp \{i:\lambda_i(\varepsilon)\leq
\lambda\}.
\]

Restricting the tangential Laplace operator $\Delta_F$ to the leaves
of ${\mathcal F}$ and lifting the restrictions to the holonomy
coverings of leaves, we get the Laplacian $\Delta_x$ acting in
$C^{\infty}_c(G^x,s^{*}\Lambda T^{*}M)$. Using the assumption that
${\mathcal F}$ is Riemannian, it can be checked that, for any $x\in
M$, $\Delta_x$ is formally self-adjoint in $L^2(G^x,s^{*}\Lambda
T^{*}M)$, that, in turn, implies its essential self-adjointness in
this Hilbert space (with initial domain
$C^{\infty}_c(G^x,s^{*}\Lambda T^{*}M)$). For each $\lambda \in
{\mathbb R}$, let $E_x(\lambda)$ be the spectral projection of
$\Delta_x$, corresponding to the semi-axis $(-\infty ,\lambda ]$.
The Schwartz kernels of the operators $E_x(\lambda)$ define a
leafwise smooth section $e_\lambda$ of the bundle $\cL(\Lambda
T^{*}M)$ over $G$.

We introduce the spectrum distribution function $N_{\mathcal
F}(\lambda)$ of the operator $\Delta_F$ by
\[
N_{\mathcal F}(\lambda)=\int_M \Tr e_\lambda(x)\,dx,\quad
\lambda\in {\mathbb R},
\]
where $dx$ denotes the Riemannian volume form on $M$. By
\cite{tang}, for any $\lambda  \in {\mathbb R}$, the function $\Tr
e_\lambda $ is a bounded measurable function on $M$, therefore,
the spectrum distribution function $N_{\mathcal F}(\lambda)$ is
well-defined and takes finite values.

As above, one can show that the family $\{E_x(\lambda): x\in M\}$
defines an element $E(\lambda)$ of the twisted von Neumann
foliation algebra $W^{*}(G,\Lambda T^{*}M)$, the holonomy
invariant transverse Riemannian volume form for $\cF$ defines a
trace $\tr_\cF$ on $W^{*}(G,\Lambda T^{*}M)$, and the right hand
side of the last formula can be interpreted as the value of this
trace on $E(\lambda)$.

\begin{theorem}[\cite{adiab}]
\label{intr} Let $(M,{\mathcal F})$ be a Riemannian foliation,
equipped with a bundle-like Riemannian metric $g$. Then we have
\[
N_{\varepsilon}(\lambda ) =\varepsilon^{-q}
\frac{(4\pi)^{-q/2}}{\Gamma((q/2)+1)} \int_{-\infty}^{\lambda}
(\lambda - \tau )^{q/2}d_{\tau}N_{\mathcal F}(\tau
)+o(\varepsilon^{-q}),\quad \varepsilon\rightarrow 0.
\]
\end{theorem}

\subsection{Riemannian submersions}\label{s:Riem-sub}
Suppose that the foliation $\cF$ is given by the fibers of a
fibration $p: M\to B$ over a compact manifold $B$. Then, for any
$x\in M$, $N^*_x\cF$ coincides with the image of the cotangent map
$dp(x)^* : T_{p(x)}^*B \to T_x^*M$. The inverse maps $(dp(x)^*)^{-1}
: N_x^*\cF \to T_{p(x)}^*B$ determine a fibration $N^*\cF \to T^*B$
whose fibers are the leaves of the linearized foliation $\cF_N$.
Thus, $N^*\cF$ is diffeomorphic to the fiber product
\[
M\times_B T^*B=\{(x,\eta)\in M\times T^*B : p(x)=y=\pi_B(y,\eta)\}
\]
with a diffeomorphism $M\times_B T^*B
\stackrel{\cong}{\longrightarrow} N^*\cF$, given by
\begin{equation}\label{e:nf}
(x,\eta)\in M\times_B T^*B\mapsto dp(x)^*(\eta)\in N^*_x\cF.
\end{equation}

A Riemannian metric $g_M$ on $M$ is bundle-like if and only if there
exists a Riemannian metric $g_B$ on $B$ such that, for any $x\in M$,
the restriction of the tangent map $dp(x) : T_xM\to T_{p(x)}B$ to
$H_x\subset T_xM$ induces an isometry from $(H_x, g_H)$ to
$(T_{p(x)}B, g_B)$, or, equivalently, $p : (M,g_M)\to (B,g_B)$ is a
Riemannian submersion.

The holonomy groupoid $G$ of $\cF$ is the fiber product
\[
M\times_BM=\{(x_1,x_2)\in M\times M : p(x_1)=p(x_2)\},
\]
where $s(x_1,x_2)=x_2, r(x_1,x_2)=x_1$. Similarly, the holonomy
groupoid $G_{\cF_N}$ is the fiber product
$N^*\cF\times_{T^*B}N^*\cF$, which consists of all
$(\nu_{x_1},\nu_{x_2})\in N^*_{x_1}\cF \times N^*_{x_2}\cF$ such
that $(x_1,x_2)\in M\times_BM$ and $(dp(x_1)^*)^{-1}(\nu_{x_1})
=(dp(x_2)^*)^{-1}(\nu_{x_2})$, with
$s_N(\nu_{x_1},\nu_{x_2})=\nu_{x_2},
r_N(\nu_{x_1},\nu_{x_2})=\nu_{x_1}$. On the other hand,
$N^*\cF\times_{T^*B}N^*\cF$ is also diffeomorphic to the fiber
product
\[
G\times_B T^*B = \{(x_1,x_2,\eta)\in M \times M\times T^*B :
p(x_1)=p(x_2)=y=\pi_B(y,\eta)\}.
\]
A diffeomorphism $G\times_B T^*B \stackrel{\cong}{\longrightarrow}
N^*\cF\times_{T^*B}N^*\cF$ can be defined as
\[
(x_1,x_2,\eta)\in G\times_B T^*B \mapsto
(dp(x_1)^*(\eta),dp(x_2)^*(\eta))\in N^*\cF\times_{T^*B}N^*\cF.
\]

For simplicity, we will consider scalar operators. For any
$(y,\eta)\in T^*B$, let $\Psi^{-\infty}((N^*\cF)_{(y,\eta)})$ be the
involutive algebra of all smoothing operators, acting on
$C^\infty((N^*\cF)_{(y,\eta)})$, where $(N^*\cF)_{(y,\eta)}$ is the
fiber of the fibration $N^*\cF\to T^*B$ at $(y,\eta)$. Consider a
field $\Psi^{-\infty}(N^*\cF)$ of involutive algebras on $T^*B$
whose fiber at $\eta\in T^*B$ is
$\Psi^{-\infty}((N^*\cF)_{(y,\eta)})$. For any section $K$ of the
field $\Psi^{-\infty}(N^*\cF)$, the Schwartz kernels of the
operators $K_{(y,\eta)}$ in $C^\infty((N^*\cF)_{(y,\eta)})$
determine a well-defined function $\sigma_K$ on $G_{{\mathcal
F}_N}\cong N^*\cF \times_{T^*B} N^*\cF$. We say that the section $K$
is smooth, if the corresponding function $\sigma_K$ is smooth. This
defines an involutive algebra isomorphism of
$C^\infty(B,\Psi^{-\infty}(N^*\cF))$ with $C^{\infty}(G_{{\mathcal
F}_N})$, where the structure of involutive algebra on
$C^{\infty}(G_{{\mathcal F}_N})$ is given by the fiberwise
composition and the fiberwise adjoint. Finally, for any smooth
section $K$ of $\Psi^{-\infty}(N^*\cF)$, its trace is given by the
formula
\[
\operatorname{tr}_{{\mathcal F}_N}(K)=\int_{T^*B}\operatorname{tr}
K_{(y,\eta)} dy\,d\eta.
\]

The principal symbol of the operator $\Delta_\varepsilon$ on
functions is a tangentially elliptic operator in
$C^{\infty}(N^*{\mathcal F})\cong C^\infty(M\times_B T^*B)$ given by
\[
\sigma(\Delta_\varepsilon) = g^{T^*B}+ \Delta_{M/B}\otimes
\operatorname{Id},
\]
where the vertical Laplace operator $\Delta_{M/B}$ is given by the
smooth family of Laplace operators along the fibers of $p$.

In this case, the operator $\Delta_\varepsilon$ can be represented,
at least, locally over the base, as a differential operator on the
base with operator-valued coefficients, and it looks likely that in
this case the formula \eqref{e:adiab} can be proved, using suitable
pseudodifferential calculus, for instance, pseudodifferential
operators with operator-valued coefficients \cite{Bal} or adiabatic
pseudodifferential calculus of Melrose
\cite{Moroianu04a,Moroianu04b}.

\subsection{Linear foliations on the torus}
Here we discuss the example of a linear foliation on the $2$-torus.
So consider the two-di\-men\-si\-o\-nal torus
$\mathbb{T}^2=\mathbb{R}^2/\mathbb{Z}^2$ with the coordinates
$(u,v)\in \mathbb{R}^2$, taken modulo integer translations, endowed
with the Euclidean metric $g=d u^2+d v^2$.

Let $\widetilde{X}$ be the vector field on $\mathbb{R}^2$ given by
\[
\widetilde{X}=\frac{\partial}{\partial u}+\alpha
\frac{\partial}{\partial v},
\]
where $\alpha\in \mathbb R$. Since $\widetilde{X}$ is translation
invariant, it determines a vector field $X$ on $\mathbb{T}^2$. The
orbits of $X$ define a one-dimensional foliation $\mathcal F$ on
$\mathbb{T}^2$. The leaves of $\mathcal F$ are the images of the
parallel lines $\widetilde{L}_{(u_0,v_0)}=\{(u_{0}+t, v_{0}+t\alpha)
:t\in\mathbb{R}\}$, parameterized by $(u_0,v_0)\in \mathbb R^2$,
under the projection $\mathbb{R}^2 \rightarrow \mathbb{T}^2$.

The Riemannian metric $g_{\varepsilon}$ on $\mathbb{T}^2$ defined by
(\ref{e:gh}) is given by
\[
g_\varepsilon=-\frac{1}{1+\alpha^2}\left(du+\alpha dv \right)^2
-\frac{\varepsilon^{-2}}{1+\alpha^2}\left( -\alpha du +dv \right)^2.
\]
The Laplace-Beltrami operator $\Delta_\varepsilon$ defined by
$g_\varepsilon$ has the form
 \[
\Delta_\varepsilon=-\frac{1}{1+\alpha^2}\left(
\frac{\partial}{\partial u}+\alpha\frac{\partial}{\partial v}
\right)^2 -\frac{\varepsilon^{2}}{1+\alpha^2}\left(
-\alpha\frac{\partial}{\partial u}+\frac{\partial}{\partial v}
\right)^2.
\]

\begin{rem}
The operator $\Delta_\varepsilon $ has the complete orthogonal
system of eigenfunctions $\{f_{kl}\in C^\infty({\mathbb{T}^2}) : (k,
l) \in \ZZ^2 \}$ given by
\[
f_{kl}(u,v)=e^{2\pi i(ku+lv)}, \quad (u,v)\in {\mathbb{T}^2},
\]
with the corresponding eigenvalues
\[
\lambda_{kl}(\varepsilon)= (2\pi)^2
\left(\frac{1}{1+\alpha^2}(k+\alpha
l)^2+\frac{\varepsilon^2}{1+\alpha^2}(-\alpha k+l)^2\right).
\]
The eigenvalue distribution function of $\Delta_\varepsilon$ has the
form
\[
N_\varepsilon(\lambda)= \# \{(k,l)\in {\mathbb Z}^2 : (2\pi)^2
\left(\frac{1}{1+\alpha^2}(k+\alpha
l)^2+\frac{\varepsilon^2}{1+\alpha^2}(-\alpha k+l)^2\right) <
\lambda\}.
 \]
Thus we see that $N_\varepsilon(\lambda)$ equals the number of
integer points in the ellipse
\[
\{(\xi,\eta)\in \mathbb{R}^2: (2\pi)^2
\left(\frac{1}{1+\alpha^2}(\xi+\alpha
\eta)^2+\frac{\varepsilon^2}{1+\alpha^2}(-\alpha
\xi+\eta)^2\right)<\lambda\}.
\]
This ellipse is centered at the origin. Its semi-axes are equal to
\[
a=\frac{(1+\alpha^2)\sqrt{\lambda}}{2\pi}, \quad
b=\frac{(1+\alpha^2)\sqrt{\lambda}}{2\pi\varepsilon}.
\]
Thus, $a$ is independent of $\varepsilon$, and $b\to \infty$ as
$\varepsilon\to 0$.

So we see that, in this case, our problem is related with some
lattice point distribution problems.
\end{rem}

In this case, by a direct calculation, one can show the following
result.

\begin{theorem}[\cite{torus-eng}]\label{th:main}
The following asymptotic formula for the eigenvalue distribution
function $N_\varepsilon(\lambda)$ of the operator
$\Delta_\varepsilon$ for a fixed $\lambda\in\mathbb{R}$ holds:
\medskip\par
1. For $\alpha\not\in\mathbb{Q},$
\begin{equation*}
    N_{\varepsilon}(\lambda ) =\frac{1}{4\pi}\varepsilon^{-1}\lambda+o(\varepsilon^{-1}), \quad
\varepsilon\rightarrow 0.
\end{equation*}

2. For $\alpha\in\mathbb{Q}$ of the form $\alpha=\frac{p}{q}$, where
$p$ and $q$ are coprime,
\begin{equation*}
    N_{\varepsilon}(\lambda )=\varepsilon^{-1}
\sum_{\substack{k\in {\mathbb Z}\\
|k|<\frac{\sqrt{\lambda}}{2\pi}\sqrt{p^2+q^2}}}
\frac{1}{\pi\sqrt{p^2+q^2} }(\lambda -
\frac{4\pi^2}{p^2+q^2}k^2)^{1/2}+o(\varepsilon^{-1}), \quad
\varepsilon\rightarrow 0.
\end{equation*}
\end{theorem}

One can also derive the asymptotic formulae of Theorem~\ref{th:main}
from Theorem~\ref{intr} (see \cite{torus-eng} for details). By
Theorem~\ref{th:main}, we have, for $\alpha\not\in\mathbb{Q}$,
\begin{multline*}
\operatorname{tr} e^{-t\Delta_\varepsilon}=\int_0^{+\infty}
e^{-t\lambda} dN_{\varepsilon}(\lambda
)=\frac{1}{4\pi}\varepsilon^{-1} \int_0^{+\infty} e^{-t\lambda}
d\lambda + o(\varepsilon^{-1})\\ = \frac{1}{4\pi t}\varepsilon^{-1}
+ o(\varepsilon^{-1}), \quad \varepsilon\rightarrow 0.
\end{multline*}

As an illustration, let us show that, for $\alpha\not\in\mathbb{Q}$,
we have
\[
\operatorname{tr}_{{\mathcal F}_N} e^{-t\sigma(\Delta_\varepsilon)}
= \frac{1}{2t}, \quad t>0,
\]
that agrees with \eqref{e:adiab}. It is easy to see that
$F=\operatorname{span}(U_1), H=\operatorname{span}(U_2)$, where
 \[
 U_1=
 \frac{1}{\sqrt{1+\alpha^2}}\left(\frac{\partial}{\partial u}+\alpha \frac{\partial}{\partial
 v}\right),\quad
 U_2= \frac{1}{\sqrt{1+\alpha^2}}\left(-\alpha \frac{\partial}{\partial u}+\frac{\partial}{\partial
 v}\right).
 \]
Therefore, we have $F^*=\operatorname{span}(\omega_1)$, $
H^*=\operatorname{span}(\omega_2)$, where
\[
\omega_1=\frac{1}{\sqrt{1+\alpha^2}}(du +\alpha dv), \quad
\omega_2=\frac{1}{\sqrt{1+\alpha^2}}(-\alpha du+dv).
\]

The conormal bundle $N^*\cF$ is canonically isomorphic to $H^*$ and,
therefore, $N^*\cF$ is diffeomorphic to $\TT^2\times\RR$ by means of
the diffeomorphism
\[
(u,v,p_2)\in \TT^2\times\RR\mapsto p_2\omega_2(u,v)\in N^*\cF.
\]
The leaves of the lifted foliation ${\cF}_N$ coincide with the
orbits of the induced flow on $N^*\cF\cong \TT^2\times\RR$ given by
\[
T_\tau(u,v,p_2)=(u+\tau, v+\alpha \tau, p_2), \quad (u,v,p_2)\in
\TT^2\times\RR.
\]

The foliation has no holonomy, and there is a natural identification
$\tilde{L}_{(u, v, p_2)}\to L_{(u, v)}$ given by
\[
(u+\tau, v+\alpha \tau, p_2) \mapsto (u+\tau, v+\alpha \tau), \quad
\tau \in \RR.
\]
Therefore, the lift $\Delta_{{\mathcal F}_N}$ of the tangential
Laplacian $\Delta_F$ to a tangentially elliptic (relative to
${\mathcal F}_N$) operator in $C^{\infty}(N^*{\mathcal F})$ is given
by
\[
\Delta_{{\mathcal
F}_N}=-\frac{1}{1+\alpha^2}\left(\frac{\partial}{\partial u}+\alpha
\frac{\partial}{\partial v}\right)^2.
\]

The induced metric $g_N$ in the fibers of $N^*\cF$ is given by
\[
g^N_{(u,v)}(p_2)=p_2^2, \quad (u,v,p_2)\in N^*\cF.
\]
Thus, the principal symbol of $\Delta_\varepsilon$ is a tangentially
elliptic operator in $C^{\infty}(N^*{\mathcal F})$ given by
\[
\sigma(\Delta_\varepsilon) =
-\frac{1}{1+\alpha^2}\left(\frac{\partial}{\partial u}+\alpha
\frac{\partial}{\partial v}\right)^2+p_2^2.
\]

The foliation $\cF$ is given by the orbits of a free action of $\RR$
on $\TT^2$, and its holonomy groupoid is described as follows:
$G=\TT^2\times \RR$, $G^{(0)}=\TT^2$, $s(u,v,\tau)=(u-\tau,v-\alpha
\tau)$, $r(u,v,\tau)=(u,v)$, $(u,v)\in \TT^2, \tau\in\RR,$ and the
product of $(u_1,v_1,\tau_1)$ and $(u_2,v_2,\tau_2)$ is defined if
$u_2=u_1-\tau_1, v_2=v_1-\alpha \tau_1$ and equals
\[
(u_1,v_1,\tau_1)(u_2,v_2,\tau_2)=(u_1,v_1,\tau_1+\tau_2).
\]

The holonomy groupoid $G_{\cF_N}$ is described as follows:
$G_{\cF_N}=\TT^2\times \RR^2$, $G_{\cF_N}^{(0)}=\TT^2\times \RR$,
$s_N(u,v,p_v,\tau)=(u-\tau,v-\alpha \tau,p_v)$,
$r_N(u,v,p_v,\tau)=(u,v,p_v)$, $(u,v,p_v)\in \TT^2\times \RR,
\tau\in\RR,$ and the product of $(u_1,v_1,p_{v,1},\tau_1)$ and
$(u_2,v_2,p_{v,2},\tau_2)$ is defined if $u_2=u_1-\tau_1,
v_2=v_1-\alpha \tau_1, p_{v,1}=p_{v,2}(=p_v),$ and equals
\[
(u_1,v_1,p_v,\tau_1)(u_2,v_2,p_v,\tau_2)=(u_1,v_1,p_v,\tau_1+\tau_2).
\]

The restriction of the operator $\sigma(\Delta_\varepsilon)$ to
$\tilde{L}_{(u,v,p_2)} \cong \RR$ is the second order elliptic
differential operator in the space
$L^2(\RR,\sqrt{1+\alpha^2}d\tau)$:
\[
\sigma(\Delta_\varepsilon)_{(u,v,p_2)} =
-\frac{1}{1+\alpha^2}\frac{\partial^2}{\partial \tau^2}+p_2^2.
\]
The change of variables $\sigma = \tau \sqrt{1+\alpha^2}$ moves this
operator to the space $L^2(\RR,d\sigma)$, getting
\[
\sigma(\Delta_\varepsilon)_{(u,v,p_2)} = -\frac{\partial^2}{\partial
\sigma^2}+p_2^2.
\]
So the heat kernel of the operator
$\sigma(\Delta_\varepsilon)_{(u,v,p_2)}$ in
$L^2(\RR,\sqrt{1+\alpha^2}d\tau)$ is given by
\[
K_t(\tau_1,\tau_2)=(4\pi t)^{-1/2}e^{-p^2_2t}
\exp\left(-\frac{(\tau_1-\tau_2)^2}{4t(1+\alpha^2)}\right).
\]
The corresponding element $k_t$ of $C^\infty(G_{\cF_N})$ such that
\[
e^{-t\sigma(\Delta_\varepsilon)_{(u,v,p_2)}}=R_{(u,v,p_2)}(k_t)
\]
is related with $K_t$ by
\[
K_t(\tau_1,\tau_2)=k_t((u,v,p_2,\tau_1)^{-1}(u,v,p_2,\tau_2)),
\]
where we consider $(u,v,p_2,\tau_1)$ and $(u,v,p_2,\tau_2)$ as
elements of $G_{\cF_N}$ (cf. \eqref{e:Rnu}). Therefore, it is given
by
\begin{align*}
k_t(u,v,p_v,\tau)& =K_t(0,\tau)\\ & =(4\pi t)^{-1/2}e^{-p^2_2t}
\exp\left(-\frac{\tau^2}{4t(1+\alpha^2)}\right), \quad
(u,v,p_v,\tau)\in G_{\cF_N}.
\end{align*}
Putting $\tau=0$, we get a well-defined function $k_t$ on $N^*\cF$,
the restriction to the set of units $G^{(0)}_{\cF_N} = N^*\cF\cong
\TT^2\times\RR$:
\[
k_t(u,v,p_2)=(4\pi t)^{-1/2}e^{-p^2_2t}, \quad (u,v,p_2)\in
\TT^2\times\RR.
\]
Finally, by \eqref{e:trfn}, we obtain
\[
\operatorname{tr}_{{\mathcal F}_N} e^{-t\sigma(\Delta_\varepsilon)}
=\int_{\TT^2\times\RR} k_t(u,v,w,p_2)\,du\,dv\,dp_2 = \frac{1}{2t}.
\]

\section{Some examples of non-Riemannian
foliations} \label{s:nonRiem}

In this Section we consider two examples of non-Riemannian
foliations. We start with one-dimensional foliations given by the
orbits of invariant flows on the three-dimensional Heisenberg
manifold.

\subsection{Riemannian Heisenberg manifolds}\label{s:Heis0}
Recall that the real three-dimensional Heisenberg group $H_3$ is the
Lie subgroup of $\operatorname{GL}(3,\mathbb{R})$ consisting of all
matrices of the form
\[
\gamma (u,v,w)=\begin{bmatrix}
 {1} & {u} & {w} \\
 {0} & {1} & {v}  \\
 0 & 0 & {1} \\
\end{bmatrix}, \quad u,v,w\in\mathbb{R}.
\]
Its Lie algebra $\mathfrak{h}_3$ is the Lie subalgebra of
$gl(3,\mathbb{R})$ consisting of all matrices of the form
\[
X(u,v,w)=\begin{bmatrix}
{0} & {u} & {w} \\
{0} & {0} & {v}  \\
 0 & 0 & {0} \\
\end{bmatrix}, \quad u,v,w\in\mathbb{R}.
\]

Denote by
\[
U=X(1,0,0),\quad V=X(0,1,0),\quad W=X(0,0,1)
\]
the standard basis in the Lie algebra $\mathfrak{h}_3$. The
corresponding left-invariant vector fields on $H_3$ are given by
\[
U=\frac{\partial}{\partial u}, \quad V=\frac{\partial}{\partial
v}+u\frac{\partial}{\partial w}, \quad W=\frac{\partial}{\partial
w},\quad \gamma(u,v,w)\in H.
\]
The dual basis of left-invariant differential one-forms is given by
\[
U^*=du, \quad V^*=dv, \quad W^*=dw-udv.
\]

Consider the uniform discrete subgroup $\Gamma$ of $H_3$ defined as
\[
\Gamma =\{ \gamma (u,v,w): u,v,w \in \mathbb Z \}.
\]
A Riemannian Heisenberg manifold $M$ is defined to be a pair
$(\Gamma \backslash H_3,g)$, where $g$ is a Riemannian metric on
$\Gamma \backslash H_3$ whose lift to $H_3$ is left $H$-invariant.

It is easy to see that $g$ is uniquely determined by the value of
its lift to $H_3$ at the identity $\gamma(0,0,0)$, that is, by a
symmetric positive definite $3\times 3$-matrix.

Let $\alpha  \in \mathbb{R}$. Consider the left-invariant vector
field on $H_3$ associated with
\[
  X(1,\alpha,0)=\begin{bmatrix}
{0} & {1} & {0} \\
{0} & {0} & {\alpha}  \\
 0 & 0 & {0} \\
\end{bmatrix} \in \mathfrak{h}_3.
\]
Since $X(1,\alpha,0)$ is a left-invariant vector field, it
determines a vector field on $M=\Gamma \backslash H_3$. The orbits
of this vector field define a one-dimensional foliation $\mathcal F$
on $M$. The leaf through a point $\Gamma \gamma(u,v,w)\in M$ is
described as
\[
L_{\Gamma \gamma(u,v,w)} =\{ \Gamma \gamma(u+\tau,v+\alpha
\tau,w+\alpha \tau u+\frac{\alpha \tau^2}{2})\in \Gamma \backslash
H_3: \tau\in {\mathbb R}\}.
\]

Let us assume that $g$ corresponds to the identity $3\times
3$-matrix. Consider the adiabatic limit associated with the
Riemannian Heisenberg ma\-ni\-fold $(\Gamma \backslash H_3,g)$ and
the one-dimensional foliation $\mathcal F$. The corresponding
Laplace-Beltrami operator $\Delta_\varepsilon$ on the group $H_3$
has the form:
\begin{multline*}
\Delta_\varepsilon=-\frac{1}{1+\alpha^2}\left(
\frac{\partial}{\partial u}+\alpha\left(\frac{\partial}{\partial
v}+u\frac{\partial}{\partial w} \right) \right)^2 \\
+\varepsilon^2 \Bigg[\frac{1}{1+\alpha^2} \left(
-\alpha\frac{\partial}{\partial u}+\frac{\partial}{\partial
v}+u\frac{\partial}{\partial w} \right)^2
+\frac{\partial^2}{\partial w^2}\Bigg].
\end{multline*}

Using an explicit computation of the heat kernel on the Heisenberg
group, one can show the following asymptotic formula.

\begin{thm}[\cite{matsb}] \label{t:trace}
For any $t>0$, we have as $\varepsilon\to 0$
\begin{equation}\label{e:nil}
\operatorname{tr} e^{-t\Delta_\varepsilon}=
\frac{1}{8\pi^2\varepsilon^{2}}\int_{-\infty}^{+\infty}
\frac{\eta}{\sinh(t\eta)} e^{-t\eta^2} d\eta + o(\varepsilon^{-2}).
\end{equation}
\end{thm}

Now we show that the formula (\ref{e:nil}) can be written in the
form (\ref{e:ncWeyl}).

\begin{thm}\label{t:nil}
Under current assumptions, for any $t>0$, we have
\[
\operatorname{tr} e^{-t\Delta_\varepsilon}=
\frac{1}{4\pi^2\varepsilon^{2}}\operatorname{tr}_{{\mathcal F}_N}
e^{-t\sigma(\Delta_\varepsilon)}+o(\varepsilon^{-2}), \quad
\varepsilon\to 0,
\]
where $\sigma(\Delta_\varepsilon)$ is defined by
\eqref{e:principal}.
\end{thm}

\begin{proof}
We have $F=\operatorname{span}(U_1)$,
$H=\operatorname{span}(U_2,U_3)$, where
 \begin{align*}
 U_1&=\frac{1}{\sqrt{1+\alpha^2}}(U+\alpha V)=
 \frac{1}{\sqrt{1+\alpha^2}}\left[\frac{\partial}{\partial u}+\alpha \left(\frac{\partial}{\partial
v}+u\frac{\partial}{\partial
w}\right)\right],\\
 U_2&=\frac{1}{\sqrt{1+\alpha^2}}(-\alpha U+V)=
 \frac{1}{\sqrt{1+\alpha^2}}\left[-\alpha \frac{\partial}{\partial u}+\frac{\partial}{\partial
v}+u\frac{\partial}{\partial w}\right],\\
 U_3&=W=\frac{\partial}{\partial w}.
 \end{align*}
We also have $F^*=\operatorname{span}(\omega_1)$,
$H^*=\operatorname{span}(\omega_2,\omega_3)$, where
\begin{align*}
\omega_1& = \frac{1}{\sqrt{1+\alpha^2}}(U^*+\alpha
V^*)=\frac{1}{\sqrt{1+\alpha^2}}(du+\alpha dv), \\
\omega_2&=\frac{1}{\sqrt{1+\alpha^2}}(-\alpha
U^*+V^*)=\frac{1}{\sqrt{1+\alpha^2}}(-\alpha
du+dv), \\
\omega_3&=dw-udv.
\end{align*}

The conormal bundle $N^*\cF$ is canonically isomorphic to $H^*$ and,
therefore, its fiber at $\Gamma\gamma(u,v,w)$ consists of all
$p_2\omega_2(u,v,w)+p_2\omega_3(u,v,w)\in T^*_{\gamma(u,v,w)}H_3$
with $p_2,p_3\in\RR$. So $N^*\cF$ is the quotient of
$H_3\times\RR^2$ by the induced action of $\Gamma$. The leaves of
the lifted foliation ${\cF}_N$ coincide with the orbits of the
induced flow on $N^*\cF$, which given by:
\[
T_\tau(u,v,w,p_2,p_3)=(u+\tau,v+\alpha \tau,w+\alpha \tau
u+\frac{\alpha \tau^2}{2},p_2-\tau\sqrt{1+\alpha^2}p_3,p_3).
\]

The foliation has no holonomy, and there is a natural identification
of the leaf $\tilde{L}_{(u,v,w,p_2,p_3)}$ of $\cF_N$ through
$\Gamma(u,v,w,p_2,p_3)\in N^*\cF$ with the leaf $L_{\gamma(u,v,w)}$
of $\cF$ through $\Gamma \gamma(u,v,w)\in M$ given by
\[
(u+\tau,v+\alpha \tau,w+\alpha \tau u+\frac{\alpha
\tau^2}{2},p_2-\tau\sqrt{1+\alpha^2}p_3,p_3)\mapsto (u+\tau,v+\alpha
\tau,w+\alpha \tau u+\frac{\alpha \tau^2}{2}).
\]
Therefore, the lift $\Delta_{{\mathcal F}_N}$  of the tangential
Laplacian $\Delta_F$ to a tangentially elliptic (relative to
${\mathcal F}_N$) operator in $C^{\infty}(N^*{\mathcal F})$ is given
by
\[
\Delta_{{\mathcal
F}_N}=-\frac{1}{1+\alpha^2}\left(\frac{\partial}{\partial u}+\alpha
\left(\frac{\partial}{\partial v}+u\frac{\partial}{\partial
w}\right)-\sqrt{1+\alpha^2}p_3\frac{\partial}{\partial p_2}
\right)^2.
\]

The induced metric $g^N$ in the fibers of $N^*\cF$ is given by
\[
g^N_{(u,v,w)}(p_2,p_3)=p_2^2+p^2_3, \quad (u,v,w,p_2,p_3)\in N^*\cF.
\]
Thus, the principal symbol of $\Delta_\varepsilon$ is a tangentially
elliptic operator in $C^{\infty}(N^*{\mathcal F})$ given by
\begin{equation}\label{e:nilsigma}
\sigma(\Delta_\varepsilon) =
-\frac{1}{1+\alpha^2}\left(\frac{\partial}{\partial u}+\alpha
\left(\frac{\partial}{\partial v}+u\frac{\partial}{\partial
w}\right)-\sqrt{1+\alpha^2}p_3\frac{\partial}{\partial p_2}
\right)^2 +p_2^2+p^2_3.
\end{equation}

The restriction of the operator $\sigma(\Delta_\varepsilon)$ to
$\tilde{L}_{(u,v,w,p_2,p_3)}\cong \RR$ is the second order elliptic
differential operator in the space
$L^2(\RR,\sqrt{1+\alpha^2}d\tau)$:
\[
\sigma(\Delta_\varepsilon)_{(u,v,w,p_2,p_3)} =
-\frac{1}{1+\alpha^2}\frac{\partial^2}{\partial
\tau^2}+(p_2-\tau\sqrt{1+\alpha^2}p_3)^2+p^2_3.
\]
The change of variable $\sigma=\tau\sqrt{1+\alpha^2}$ transfers this
operator to the space $L^2(\RR,d\sigma)$, getting
\[
\sigma(\Delta_\varepsilon)_{(u,v,w,p_2,p_3)} =
-\frac{\partial^2}{\partial \sigma^2}+(p_2-\sigma p_3)^2+p^2_3.
\]
Now we use the well-known Mehler formula for the heat kernel $H_t$
of the harmonic oscillator $H_\omega=-\frac{\partial^2}{\partial
x^2}+\omega^2x^2$ in $L^2(\RR, dx)$:
\[
H^\omega_t(x,y)=(4\pi t)^{-1/2}\left(\frac{2\omega t}{\sinh(2\omega
t)} \right)^{1/2}\exp\left(-\frac{\frac{2\omega t}{\sinh(2\omega
t)}[\cosh(2\omega t)(x^2+y^2)-2xy]}{4t}\right).
\]
So the heat kernel of the operator
$\sigma(\Delta_\varepsilon)_{(u,v,w,p_2,p_3)}$ in
$L^2(\RR,\sqrt{1+\alpha^2}d\tau)$ is given by
\[
K_t(\tau_1,\tau_2)=
H^{p_3}_t\left(\tau_1\sqrt{1+\alpha^2}-\frac{p_2}{p_3},\tau_2\sqrt{1+\alpha^2}
-\frac{p_2}{p_3}\right).
\]

The foliation $\cF$ is given by the orbits of a free action of $\RR$
on $M$, and its holonomy groupoid is described as follows:
$G=M\times \RR$, $G^{(0)}=M$, for any $(u,v,w,\tau)\in G$,
$s(u,v,w,\tau)=(u-\tau,v-\alpha \tau,w-\alpha \tau u+\frac{\alpha
\tau^2}{2})$ and $r(u,v,w,\tau)=(u,v,w)$. The product of
$(u_1,v_1,w_1,\tau_1)$ and $(u_2,v_2,w_2,\tau_2)$ is defined if
$u_2=u_1-\tau_1, v_2=v_1-\alpha \tau_1, w_2=w_1-\alpha \tau_1
u_1+\frac{\alpha \tau_1^2}{2}$ and equals $(u_1,v_1,w_1,\tau_1)
(u_2,v_2,w_2,\tau_2)=(u_1,v_1,w_1, \tau_1+\tau_2)$.

The holonomy groupoid $G_{\cF_N}$ is described as follows:
$G_{\cF_N}=M\times \RR^2\times \RR$, $G_{\cF_N}^{(0)}=M\times
\RR^2$, and for any $(u,v,w,p_2,p_3,\tau)\in G_{\cF_N}$, we have
$s_N(u,v,w,p_2,p_3,\tau)=(u-\tau,v-\alpha \tau,w-\alpha \tau
u+\frac{\alpha \tau^2}{2},p_2+\tau\sqrt{1+\alpha^2}p_3,p_3)$, and
$r_N(u,v,w,p_2,p_3,\tau)=(u,v,w,p_2,p_3)$, .

The corresponding element $k_t$ of $C^\infty(G_{\cF_N})$ is given by
\[
k_t(u,v,w,p_2,p_3,\tau)=K_t(0,\tau)
=H^{p_3}_t\left(-\frac{p_2}{p_3},\tau \sqrt{1+\alpha^2}
-\frac{p_2}{p_3}\right).
\]

Putting $\tau=0$, we get a well-defined function $k_t$ on $N^*\cF$,
the restriction to the set of units $G^{(0)}_{\cF_N} = N^*\cF$:
\[
k_t(u,v,w,p_2,p_3)=(4\pi t)^{-1/2}\left(\frac{2p_3 t}{\sinh(2p_3 t)}
\right)^{1/2} e^{-p^2_3t}\exp\left(-\frac{\sinh(p_3 t)}{\cosh(p_3
t)}\frac{p^2_2}{p_3} \right).
\]

To compute the Liouville transverse volume form, we choose a
foliated chart defined by
\[
x_1=u, \quad y_1=v-\alpha u, \quad y_2=w-\frac12\alpha u^2.
\]
Thus, we have
\[
dy_1=dv-\alpha du=\sqrt{1+\alpha^2}\omega_2, \quad dy_2=dw-\alpha
u\,du=\omega_3+\sqrt{1+\alpha^2}u\omega_2.
\]
The identity
\[
\eta_1dy_1+\eta_2dy_2=p_2\omega_2+p_3\omega_3
\]
implies a relation between the dual coordinates $(\eta_1,\eta_2)$ in
this foliated chart and $(p_2,p_3)$:
\[
\eta_1=\frac{1}{\sqrt{1+\alpha^2}}p_2-up_3,\quad \eta_2=p_3.
\]
So we see that the combination of the transverse Liouville measure
$|dy_1\wedge dy_2\wedge d\eta_1\wedge d\eta_2|$ with the leafwise
Riemannian volume density $|\omega_1|$ determines the measure $d\nu$
on $N^*\cF$ given by
\[
d\nu =du\,dv\,dw\,dp_2\,dp_3.
\]

By \eqref{e:trfn}, it follows that
\begin{align*}
\operatorname{tr}_{{\mathcal F}_N} e^{-t\sigma(\Delta_\varepsilon)}
& =\int k_t(u,v,w,p_2,p_3)\,du\,dv\,dw\,dp_2\,dp_3\\ & = (4\pi
t)^{-1/2}\iint \left(\frac{2p_3 t}{\sinh(2p_3 t)} \right)^{1/2}
e^{-p^2_3t}\exp\left(-\frac{\sinh(p_3 t)}{\cosh(p_3
t)}\frac{p^2_2}{p_3} \right)\,dp_2\,dp_3.
\end{align*}
Integrating with respect to $p_2$ in the last integral and using the
formula
\[
\int e^{-ax^2}\,dx=\frac{\pi^{1/2}}{a^{1/2}},
\]
we obtain
\begin{align*}
\operatorname{tr}_{{\mathcal F}_N} e^{-t\sigma(\Delta_\varepsilon)}
& = \frac{1}{2}\int \left(\frac{2p_3}{\sinh(2p_3 t)} \right)^{1/2}
\left(\frac{p_3\cosh(p_3 t)}{\sinh(p_3 t)}\right)^{1/2}
e^{-p^2_3t}\,dp_3\\
& = \frac{1}{2}\int \frac{p_3}{\sinh(p_3 t)} e^{-p^2_3t}\,dp_3.
\end{align*}
Combined with (\ref{e:nil}), this concludes the proof.
\end{proof}

\subsection{Sol-manifolds} In this Section we consider another
example of adiabatic limits associated with non-Rie\-mannian
foliations, namely, with one-dimensional foliations given by the
orbits of invariant flows on Riemannian Sol-manifolds.

Let $A=\begin{pmatrix} a_{11} & a_{12}\\ a_{21} & a_{22}
\end{pmatrix}\in
    \operatorname{SL}(2,\mathbb{Z})$ such that $|\operatorname {tr} A|>2$.
A Sol-manifold can be defined as the quotient
$M^3_A=(\mathbb{T}^2\times\mathbb{R}) /\mathbb{Z}$, where the action
of $\mathbb{Z}$ on $\mathbb{T}^2\times\mathbb{R}$ is defined by the
diffeomorphism $T_A$ of $\mathbb{T}^2\times\mathbb{R}$ given by
\[
        T_A(x, y, z)= (a_{11}x+a_{12}y, a_{21}x+a_{22}y,
        z+1), \quad (x,y)\in \mathbb{T}^2= {\mathbb R}^2/{\mathbb Z}^2, \quad
z\in \mathbb{R}.
\]
It is well known that $M^3_A$ is a compact manifold.

Denote by $\lambda$ and $\lambda^{-1}$ the eigenvalues of $A$ and
suppose that $\lambda>1$. Let $\{e_u,e_v\}$ be the corresponding
positively oriented base of eigenvectors of $A$. We introduce
another coordinate system $(u,v,w)$ on $\RR^3$ by
\[
(x,y)=u e_u+v e_v, \quad w=z\ln \lambda.
\]
In the coordinates $(u,v,w)$ the map $T_A$ takes the following form:
\[
        T_A(u, v, w)= (\lambda u,
        \lambda^{-1}v, w+\ln \lambda). \]

Denote by $\widetilde M^3$ the Lie subgroup of the Lie group
$\operatorname{GL}(3, \mathbb {R})$ of matrices of the form
\[
\gamma(u,v,w)=\begin{pmatrix}
 e^w & 0      &u\\
 0   & e^{-w} &v\\
 0   & 0      &1
\end{pmatrix}, \quad u,v,w\in {\mathbb R}.
\]
The corresponding Lie algebra $\mathfrak{m}^3$ of $\widetilde{M}^3$
is the Lie subalgebra of the Lie algebra
$\operatorname{gl}(3,\mathbb{R})$ of matrices of the form
\[
X(u,v,w)=\begin{pmatrix}
 w & 0      &u\\
 0   & -w &v\\
 0   & 0      &0
\end{pmatrix}, \quad u,v,w\in {\mathbb R}.
\]
Denote by
\[
U=X(1,0,0),\quad V=X(0,1,0),\quad W=X(0,0,1)
\]
the standard basis in the Lie algebra $\mathfrak{m}^3$. The
corresponding left-invariant vector fields on $\widetilde{M}^3$ are
given by
\[
U=e^{w}\frac{\partial}{\partial u}, \quad
V=e^{-w}\frac{\partial}{\partial v},\quad W=\frac{\partial}{\partial
w},\quad \gamma(u,v,w)\in \widetilde{M}^3.
\]

The dual basis $U^*,V^*,W^*$ of left-invariant one-forms is given by
\[
U^*=e^{-w}du,\quad V^*=e^{w}dv,\quad W^*=dw.
\]

Observe that in the new coordinates two points $(u,v)$ and
$(u^\prime,v^\prime)$ define the same point on $\mathbb T^2$ if and
only if $(u-u^\prime,v-v^\prime)=k(c^1_1,c^2_1)+m(c^1_2,c^2_2)$,
where $k,m\in \mathbb Z$ and $\{e_1=(c^1_1,c^2_1),
e_2=(c^1_2,c^2_2)\}$ is the basis of the integer lattice $\Gamma$ of
$\mathbb{T}^2$ defined by the equation
\[
A=\begin{pmatrix} a_{11} & a_{12}\\ a_{21} & a_{22}
\end{pmatrix}=\begin{pmatrix} c^1_1 & c^1_2\\ c^2_1 & c^2_2
\end{pmatrix}^{-1}\begin{pmatrix} \lambda & 0\\ 0 & \lambda^{-1}
\end{pmatrix}\begin{pmatrix} c^1_1 & c^1_2\\ c^2_1 & c^2_2
\end{pmatrix}.
\]
Therefore, if we introduce a discrete subgroup $G_A$ of
$\widetilde{M}^3$, which consists of all $\gamma(u,v,w)\in
\widetilde{M}^3$ such that
  \[
(u,v)=k(c^1_1,c^2_1)+l(c^1_2,c^2_2)\in\Gamma,\quad w=m\ln\lambda,
\quad k,l,m\in \mathbb Z,
 \]
then it is easy to see that $M^3_A = G_A \backslash
\widetilde{M}^3$.

\begin{defn} A Riemannian Sol-manifold is a pair $(M^3_A,g)$, where
$g$ is a Riemannian metric on $M^3_A$ whose lift on
$\widetilde{M}^3$ is left $\widetilde{M}^3$-invariant.
\end{defn}

It is easy to see that $g$ is uniquely determined by its value at
the identity $\gamma(0,0,0)$ of the Lie group $\widetilde{M}^3$,
which is given by a symmetric positive definite $3\times 3$-matrix.

Suppose that $g$ corresponds to the identity matrix:
\[
g(G_A\gamma(0,0,0))=
\begin{pmatrix}
  1 & 0 & 0 \\
  0 & 1 & 0 \\
  0 & 0 & 1 \\
\end{pmatrix}.
\]
In other words, the basis $U,V,W$ is orthonormal. Then, in the
coordinates $(u,v,w)$, the metric $g$ has the form
\[
g(G_A\gamma(u,v,w))=e^{-2w}du^2+e^{2w} dv^2 +dw^2.
\]

Let $\alpha  \in \mathbb{R}$. Consider the left-invariant vector
field on $\widetilde{M}^3$ associated with
\[
  X(1,\alpha,0)=\begin{bmatrix}
{0} & 0 & {1} \\ {0} & {0} & {\alpha}  \\
 0 & 0 & {0} \\
\end{bmatrix} \in \mathfrak{m}^3.
\]
The orbits of the corresponding vector field on $M^3_A$ define a
one-dimensional foliation $\mathcal F$. The leaf $L_{G_A
\gamma(u,v,w)}$ of $\mathcal F$ through $G_A \gamma(u,v,w)\in
M^3_{A}$ has the form
\[
L_{G_A \gamma(u,v,w)} =\{ G_A\gamma(u+e^w t,v+\alpha e^{-w} t,w)\in
 {M}^3_A: t\in {\mathbb R}\}.
\]

One can show that, if $\alpha=0$, the metric $g$ is bundle-like and
the foliation $\mathcal F$ is Riemannian. Otherwise, the metric $g$
is not bundle-like.

Let us consider the adiabatic limit associated with the Riemannian
Sol-manifold $(G_A \backslash \widetilde{M}^3,g)$ and with the
foliation $\mathcal F$. Then the Riemannian metric $g_\varepsilon$
on $G_A \backslash \widetilde{M}^3$ defined by (\ref{e:gh}) in the
local coordinates $(u,v,w)$ has the form
\[
g_\varepsilon=\frac{1}{1+\alpha^2}(e^{-w}du +\alpha e^{w}dv)^2+
\varepsilon^{-2}\left[\frac{1}{1+\alpha^2}(-\alpha e^{-w}du
+e^{w}dv)^2+ dw^2\right].
\]

For any $\varepsilon>0$, consider the Laplace-Beltrami operator
$\Delta_\varepsilon$ defined by $g_\varepsilon$:
\begin{equation*}
\Delta_\varepsilon=-\frac{1}{1+\alpha^2}
\left(e^{w}\frac{\partial}{\partial u}+\alpha
e^{-w}\frac{\partial}{\partial v}\right)^2 -
\varepsilon^2\left(\frac{1}{1+\alpha^2} \left(-\alpha
e^{w}\frac{\partial}{\partial u}+e^{-w}\frac{\partial}{\partial
v}\right)^2+ \frac{\partial^2}{\partial w^2}\right).
\end{equation*}

\begin{theorem}[\cite{MZ,SMZ}]\label{t:mainsolmanifolds}
For $\lambda>0$  for the eigenvalue distribution function
$N_\varepsilon(\lambda)$ of $\Delta_\varepsilon$, we have:

\begin{itemize}
\item For $\alpha\neq 0$
\[
N_\varepsilon(\lambda)=\frac{1}{4\pi^2}\lambda^{\frac{3}{2}}\varepsilon^{-2}+o(\varepsilon^{-2}).
\]
\item For $\alpha= 0$
  \[
  N_\varepsilon(\lambda)=\frac{1}{6\pi^2}\lambda^{\frac{3}{2}}\varepsilon^{-2}+o(\varepsilon^{-2}).
  \]
\end{itemize}

\end{theorem}

For $\alpha\neq 0$ the proof of this theorem makes use of the
computation of the spectrum of the Laplace-Beltrami operator on
Sol-manifolds given in \cite{Bolsinov} (see also \cite{zuev06}),
which is a continuation of the study of geodesic flows on Riemannian
Sol-manifolds (see
\cite{TaimanovBolsinov1},\cite{TaimanovBolsinov2}), and
semiclassical Weyl formula \cite{Helffer} for the modified Mathieu
operator
\[
H_\varepsilon=-\varepsilon^2\frac{d^2}{dx^2}+a\operatorname{ch}(2\mu
x), \quad x\in {\mathbb R}.
\]
For $\alpha= 0$ the result is a consequence of
Theorem~\ref{ad:main}.

\begin{theorem}
For $\alpha\neq 0$, we have
\[
\operatorname{tr}_{{\mathcal F}_N} e^{-t\sigma(\Delta_\varepsilon)}=
\frac{1}{2}\int_{-\infty}^{+\infty} \frac{\frac{2\alpha}{1+\alpha^2}
\eta}{\sinh(\frac{2\alpha}{1+\alpha^2} t\eta )} e^{-t\eta^2}\,d\eta.
\]
\end{theorem}

\begin{proof}
It is easy to see that $F=\operatorname{span}(U_1)$ and
$H=\operatorname{span}(U_2,U_3)$, where
 \begin{align*}
 U_1&=\frac{1}{\sqrt{1+\alpha^2}}(U+\alpha V)=
 \frac{1}{\sqrt{1+\alpha^2}}\left(e^w\frac{\partial}{\partial u}+\alpha e^{-w}\frac{\partial}{\partial
 v}\right),\\
 U_2&=\frac{1}{\sqrt{1+\alpha^2}}(\alpha U- V)=
\frac{1}{\sqrt{1+\alpha^2}}\left(\alpha e^w\frac{\partial}{\partial
u}- e^{-w}\frac{\partial}{\partial
 v}\right),\\
 U_3&=W=\frac{\partial}{\partial w}.
 \end{align*}

Therefore, we have $F^*=\operatorname{span}(\omega_1)$,
$H^*=\operatorname{span}(\omega_2,\omega_3)$, where
\begin{align*}
\omega_1& = \frac{1}{\sqrt{1+\alpha^2}}(U^*+\alpha
V^*)=\frac{1}{\sqrt{1+\alpha^2}}(e^{-w}du +\alpha e^{w}dv), \\
\omega_2&=\frac{1}{\sqrt{1+\alpha^2}}(-\alpha
U^*+V^*)=\frac{1}{\sqrt{1+\alpha^2}}(-\alpha
e^{-w}du+e^{w}dv), \\
\omega_3&=W^*=dw.
\end{align*}

The conormal bundle $N^*\cF$ is canonically isomorphic to $H^*$ and,
therefore, its fiber at $\Gamma\gamma(u,v,w)$ consists of all
$p_2\omega_2(u,v,w)+p_3\omega_3(u,v,w)\in T^*_{\gamma(u,v,w)}M^3_A$
with $p_2,p_3\in\RR$. So $N^*\cF$ is the quotient of
$M^3_A\times\RR^2$ by the induced action of $\Gamma$. The leaves of
the lifted foliation ${\cF}_N$ coincide with the orbits of the
induced flow on $N^*\cF$, which given by:
\[
T_\tau(u,v,w,p_2,p_3)=(u+e^w \tau, v+\alpha e^{-w}
\tau,w,p_2+\frac{2\alpha}{\sqrt{1+\alpha^2}}\tau p_3,p_3).
\]

The foliation has no holonomy, and there is a natural identification
$\tilde{L}_{(u,v,w,p_2,p_3)}\to L_{(u,v,w)}$ given by
\[
(u+e^w \tau,v+\alpha e^{-w}
\tau,w,p_2+\frac{2\alpha}{\sqrt{1+\alpha^2}}\tau p_3,p_3) \mapsto
(u+e^w \tau,v+\alpha e^{-w} \tau,w).
\]
Therefore, the lift $\Delta_{{\mathcal F}_N}$ of the tangential
Laplacian $\Delta_F$ to a tangentially elliptic (relative to
${\mathcal F}_N$) operator in $C^{\infty}(N^*{\mathcal F})$ is given
by
\[
\Delta_{{\mathcal
F}_N}=-\frac{1}{1+\alpha^2}\left(e^{w}\frac{\partial}{\partial
u}+\alpha e^{-w}\frac{\partial}{\partial
v}+\frac{2\alpha}{\sqrt{1+\alpha^2}} p_3\frac{\partial}{\partial
p_2} \right)^2.
\]

The induced metric $g^N$ in the fibers of $N^*\cF$ is given by
\[
g^N_{(u,v,w)}(p_2,p_3)=p_2^2+p^2_3, \quad (u,v,w,p_2,p_3)\in N^*\cF.
\]
Thus, the principal symbol of $\Delta_\varepsilon$ is a tangentially
elliptic operator in $C^{\infty}(N^*{\mathcal F})$ given by
\[
\sigma(\Delta_\varepsilon) =
-\frac{1}{1+\alpha^2}\left(e^{w}\frac{\partial}{\partial u}+\alpha
e^{-w}\frac{\partial}{\partial v}+\frac{2\alpha}{\sqrt{1+\alpha^2}}
p_3\frac{\partial}{\partial p_2} \right)^2 +p_2^2+p^2_3.
\]
Thus, we see that the principal symbol has a form very similar to
the case of Heisenberg manifolds (cf. \eqref{e:nilsigma}). So we can
complete the proof, following the proof of Theorem~\ref{t:nil}.
\end{proof}

By Theorem~\ref{t:mainsolmanifolds}, we have
\[
\operatorname{tr} e^{-t\Delta_\varepsilon}=\int_{0}^{+\infty}
e^{-t\lambda}\,dN_\varepsilon(\lambda)=\frac{3}{8\pi^2\varepsilon^{2}}
\frac{\Gamma(\frac32)}{t^{3/2}}+o(\varepsilon^{-2}), \quad
\varepsilon\to 0.
\]
So we observe that in this case the formula \eqref{e:ncWeyl} does
not hold.

\end{document}